\documentclass{actams-author}
\usepackage{amssymb, amsmath,amsmath,latexsym,amssymb,amsfonts,amsbsy, amsthm,mathtools,graphicx,CJKutf8,CJKnumb,CJKulem,color,cases,empheq,extpfeil,bm}
\numberwithin{equation}{section} 

\newcommand{\dif}{\,\mathrm{d}}

\begin{document}

\abovedisplayskip 6pt plus 2pt minus 2pt \belowdisplayskip 6pt
plus 2pt minus 2pt
\def\vsp{\vspace{1mm}}
\def\th#1{\vspace{1mm}\noindent{\bf #1}\quad}
\def\proof{\vspace{1mm}\indent{\bf Proof}\quad}
\def\no{\nonumber}
\newenvironment{prof}[1][Proof]{\indent\textbf{#1}\quad }
{\hfill $\Box$\vspace{0.7mm}}
\def\q{\quad} \def\qq{\qquad}
\allowdisplaybreaks[4]


\AuthorMark{Xiaoguang. You\,\&  Aibin. Zang}                             

\TitleMark{\uppercase{Singular limit of second grade fluid}}  

\title{\uppercase{Singular limit of 2D second grade fluid past an obstacle}        
\footnote{ }}                 
\author{\sl{Xiaoguang \uppercase{You}}}    
   { School of Mathematics, Northwest University, Xi'an 710069, P. R. China\\
    E-mail\,$:wiliam\_you@aliyun.com$ }

\author{\sl{Aibin \uppercase{Zang$^{\dag}$}}}    
{School of Mathematics and Computer Science $\&$ The Center of Applied Mathematics, Yichun university, Yichun, Jiangxi, 340000 P. R. China\\
    E-mail\,$:zangab05@126.com$ }  

\maketitle%

\Abstract{In this paper, we consider the 2D second grade fluid past an obstacle satisfying the standard non-slip boundary condition at the surface of the obstacle. Second grade fluid model is a well-known non-Newtonian model, with two parameters: $\alpha$ representing length-scale, while $\nu > 0$ corresponding to viscosity. We prove that, under the constraint condition $\nu = {o}(\alpha^\frac{4}{3})$,  the second grade fluid with a suitable initial velocity converges to the Euler fluid as $\alpha$ tends to zero. Moreover, we estimate the convergence rate of the solution of second grade fluid equations to the one of Euler fluid equations as $\nu$ and $\alpha$ approach zero.}      


\MRSubClass{76A05; 35B40; 35B25}      

\section{Introduction}
Let $\mathcal{O} \subset \mathbb{R}^2$ be a bounded, simply connected domain with a smooth Jordan boundary $\Gamma$.  Set $\Omega := \mathbb{R}^2 \setminus \bar{\mathcal{O}}$ be the exterior domain. We consider the following second grade fluid equations
\begin{numcases}{}
    \partial_t \bm{v}^{\nu, \alpha} + \bm{u}^{\nu, \alpha} \cdot {\nabla} \bm{v}^{\nu, \alpha} + (\nabla \bm{u}^{\nu, \alpha})^t \cdot \bm{v}^{\nu, \alpha} + \nabla p^{\nu, \alpha} = \nu \Delta \bm{u}^{\nu, \alpha}
      &  $\text{in}  \ \Omega \times (0, \infty), $ \label{second-grade-origin-1}   \\
    \text{div} \ \bm{u}^{\nu, \alpha} = 0 &  $\text{in} \ \Omega \times [0, \infty) $ \label{second-grade-origin-2},\\
    \bm{u}^{\nu, \alpha} = 0  & $\text{in} \ \Gamma \times [0, \infty] \label{second-grade-origin-3}$,\\
    \bm{u}^{\nu, \alpha}|_{t=0} = \bm{u}_0^{\nu,\alpha} & $\text{in} \ \Omega $\label{second-grade-origin-4},\\
    \bm{u}^{\nu, \alpha}(x, t) \to 0 & $as \ |x| \to \infty $ \label{second-grade-origin-5}
\end{numcases}
where $\bm{v}^{\nu, \alpha} = \bm{u}^{\nu, \alpha} - \Delta ^{\nu, \alpha}$ and  $\bm{u}_0^{\nu,\alpha}$ is the initial velocity.

Second grade fluid models are a well-known subclass of non-Newtonian  Rivlin-Ericksen fluids of differential type \cite{rivlin1997stress}, which usually arise in  petroleum industry, polymer technology and liquid crystal suspension problems. We refer \cite{dunn1974thermodynamics,fosdick1979anomalous} for a comprehensive theory of the second grade fluids.

We mention some well known results about second grade fluid equations. In \cite{busuioc1999second}, Busuioc and Valentina studied the existence of solutions for whole space $\mathbb{R}^n$(n=2, 3). They proved that there exists a local strong solution provided the initial data is sufficient smooth, and the solution is global for two dimension. The existence, the uniqueness of a strong solution and the dynamics of second grade fluids in the torus $\mathbb{T}^2$ was studied in \cite{paicu2012regularity} by Paicu, Raugel and Rekalo, and in \cite{paicu2013dynamics} by the first two authors, using the Lagrangian approach. The existence and uniqueness of solutions  for a bounded domain with non-penetration boundary condition, was proved by Cioranescu and Ouazar \cite{cioranescu1984existence}, and we refer \cite{cioranescu1997weak,bresch1997existence} for further results. In exterior of an obstacle satisfying the non-penetration boundary condition, the global existence and uniqueness of solutions was proved in \cite{xiaoguang2021secondgrade}.

The purpose of this article is to study the asymptotic behavior of second grade fluid. In particular, A natural question to ask is whether the limiting flow satisfies the incompressible Euler equations as $\bm{u}^{\nu, \alpha}$ as $\nu, \alpha$ both tends to zero. Let's write explicitly the incompressible Euler equations as follows:
\begin{numcases}{}
\partial_t \bar{\bm{u}} + \bar{\bm{u}} \cdot \nabla \bar{\bm{u}} + \nabla \bar{p} = 0 & $\text{in} \ \Omega \times [0, \infty), $ \label{euler-1}\\
\text{div} \ \bar{\bm{u}} = 0  & $\text{in} \ \Omega \times [0, \infty), $ \label{euler-2}\\
\bar{\bm{u}} \cdot \nu = 0 & $\text{on} \ \Gamma \times [0, \infty), $ \label{euler-3}\\
\bar{\bm{u}}(\cdot, 0) = \bm{u}_0 & $\text{in} \ \Omega,$ \label{euler-4}\\
\bar{\bm{u}}(x, t) \to 0 & $as \ |x| \to \infty,  t\in[0, \infty). $ \label{euler-5}
\end{numcases}
where $\bm{u}_0$ is the smooth initial velocity. This singular limit problem has been studied in several literatures. Let us first recall that in the absence of boundaries one can obtain $\bm{H}^3$ estimates uniformly in $\nu$ and $\alpha$ in both dimensions two and three and pass to the limit. This was performed in \cite{linshiz2010convergence,busuioc2012incompressible}, see also \cite{zhou2020convergence} for a simpler proof. In 2D bounded domain with the Navier boundary conditions, the authors in \cite{busuioc2012incompressible} verified the expected convergence in 2D for weak $\bm{H}^1$ solutions . Moreover, they also proved the convergence in 3D but under the additional hypothesis that the solutions exist on a time interval independent of $\nu$ and $\alpha$. In the case of 2D bounded domain with non-slip boundary conditions, the second author and his collaborators \cite{lopes2015approximation} showed that second grade fluids converges to the incompressible Euler fluid under the constraint condition that
\begin{align}\label{constraint-condition}
\nu = O(\alpha^2).
\end{align}
 The case of non-slip boundary conditions in 3D is open.

Our interest of this article is to extends the convergence result in \cite{lopes2015approximation} to 2D exterior domain case. In other words, we will verify whether the solutions of equations \eqref{second-grade-origin-1}--\eqref{second-grade-origin-5} converges to the one of equations \eqref{euler-1}--\eqref{euler-5}. Comparing with the case of bounded domain, the major difference lies in that it's more complicated to obtain uniform bound of $\bm{u}^{\nu, \alpha}$, due to the rather involved form of the nonlinearities and high order derivatives in ($\ref{second-grade-origin-1}$). Furthermore, since 2D exterior domain is not simply connected, the existence of stream function corresponding to a general solenoidal vector fields is nontrivial. However, the stream function is the key component to construct a boundary layer corrector, which compensates for the mismatch between the non-penetration boundary condition of ideal flows and the non-penetration boundary condition of viscous flows. Once the above issues solved, we can apply the energy method to obtain the convergence rate of the solutions of the second grade fluid equations to the one of Euler fluid equations. Next, let's briefly elaborated the main results. Suppose that the following  hypothesis
\begin{equation}\label{velocity-condition} \tag{H}
\begin{split}
    &(i) \ \ \|\bm{u}_0^{\nu,\alpha} - \bm{u}_0 \|_{L^2(\Omega)} \rightarrow 0, \\
    &(ii)  \ \|D^{k} \bm{u}_0^{\nu,\alpha}\|_{L^2(\Omega)} = o(\alpha^{-k }) \text{ for } 1 \leq k \leq 3.
\end{split}
\end{equation}
holds as $\alpha \rightarrow 0$, then for arbitrarily fixed $T > 0$, we have
\begin{align}
\nonumber \|\bm{u}^{\nu, \alpha} - \bar{\bm{u}}\|_{L^\infty([0, T]; L^2(\Omega))} \leq K_T(\|\bm{u}_0^{\nu, \alpha} - \bm{u}_0\|_{L^2(\Omega)} + \alpha\|\nabla\bm{u}_0^{\alpha}\|_{L^2(\Omega)} + \alpha^\frac{1}{3} + \nu^{\frac{1}{2}}\alpha^{-\frac{2}{3}})
\end{align}
where $K_T$ is a constant only depends on $T$. From the above inequality, we conclude that, as $\alpha$ tends to zero, the solution $\bm{u}^{\nu, \alpha}$ of second grade fluid equations converges to the solution $\bar{\bm{u}}$ of the Euler equations with the constraint condition \eqref{constraint-condition} being relaxed as $\nu = o(\alpha^\frac{4}{3})$. In particular, we claim that the solutions of Euler-$\alpha$ fluid equations in 2D exterior domain with suitable initial data converges to the one of Euler flow.

The remainder of this paper is organized in three sections. In section 2,  we introduce some notations and preliminary lemmas, then we present the main theorem. In section 3, we first give a priori estimate of $\bm{u}^{\nu, \alpha}$, then construct a boundary layer corrector $\bm{u}_b$, finally proceed the proof of our main theorem by energy method. In section 4, we will examine the assumptions of the main theorem.
\section{Notations and main results}
 In this section, we will first introduce notations, then present some preliminary lemmas, finally, we state our main results.
\subsection{Notations}
We use the notation $H^s(\Omega)$ for the usual $L^2$-based Sobolev spaces of order $s$. $\bm{H}^s(\Omega)$ represents vector space $(H^s(\Omega))^2$. For the case $s = 0$, we denote $L^2(\Omega)$ the vector space $(L^2(\Omega))^2$.

For the sake of convenience,  we introduce the following notations
\begin{align*}  
  & \mathcal{D} := \lbrace \textbf{u} \in (C_0^\infty(\Omega))^2; \text{div}\, \textbf{u} = 0\rbrace,\\
  & V  := \lbrace \textbf{u} \in \bm{H}_0^1(\Omega); \ \text{div} \, \textbf{u}  =  0  \ \text{in} \ \Omega\rbrace,\\
  & L^2_{\sigma} := \lbrace \textbf{u} \in (L^2(\Omega))^2; \ \text{div}\, \textbf{u} = 0, \textbf{u} \cdot \nu|_{\Gamma} = 0 \rbrace,
\end{align*}
where $\nu$ is the nomral vector to $\Gamma$.

We will use a few times of homogeneous Sobolev space, which is stated as
\begin{align*}
& \dot{\bm{H}}(\Omega) := \{ \bm{u} \in (L_{loc}^2(\Omega))^2; \int_\Omega|\nabla \bm{u}|^2\dif x<\infty\}.
\end{align*}

For a scale function $\psi$, we denote $(-\partial_2 \psi, \partial_1 \psi)$ by $\nabla^\perp \psi$, while for a vector field $\bm{u}$, we will use notation $\nabla^\perp \cdot \bm{u} := -\partial_2 \bm{u}_1 + \partial_1 \bm{u}_2$.

Throughout the paper, if we denote by $K$ a positive constant with neither any subscript nor superscript then $K$ is considered as a generic constant whose value can change from line to line in the inequalities. On the other hand,
we denote $K_T$ a positive constant such that may depends on parameter $T$. Besides, we will denote in a bold character vector valued functions and in the usual scalar functions.

In the sequel of this subsection, we present two lemmas that would used in the following sections. In a simple connected domain, it's well known that a solenoidal field $\bm{u}$ corresponds to a stream function $\psi$ such that $\bm{u} = \nabla^\perp \psi$. While in not simple connected domain, for example 2D exterior domain, we can not take this conclusion for granted. Fortunately, Lemma 3.2 of \cite{xiaoguang2021euleralpha} established the following result.
\begin{lemma}\label{stream-function-lemma}
Suppose that $\bm{u} \in L^2_\sigma \cap \bm{H}^1(\Omega)$, then there exists a scalar function $\psi \in \dot{\bm{H}}(\Omega)$ such that $\bm{u} = \nabla ^\perp \psi$.
\end{lemma}
\begin{remark} Since $\bm{u}$ satisfies non-penetration boundary condition, it's easy to see that we can choose a particular $\psi$  that vanishes on the boundary $\Gamma$.
\end{remark}

The following well known results about stationary Stokes equations will be used in section 3. With regard to the existence and uniqueness of solutions to \eqref{stokes-1}--\eqref{stokes-2}, we refer \cite{borchers1993boundedness,galdi2011introduction}. The estimate \eqref{stokes-estimate} can be found, for example in \cite{giga2018handbook}.
\begin{lemma}\label{lemma-stokes}
Suppose $\psi  \in \bm{H}^1(\Omega)$ and $\alpha > 0$, then the stationary Stokes equation
\begin{numcases}{}
            \bm{u} - \alpha^2\Delta \bm{u} + \nabla p = \psi &  $\text{in } \Omega \label{stokes-1} $,\\
            \bm{u} = 0 & $\text{on } \Gamma$ \label{stokes-2}
\end{numcases}
has a unique solution $\displaystyle \bm{u} \in \bm{H}^3(\Omega) \cap V$ with the estimate
\begin{equation}\label{stokes-estimate}
     \begin{split}
      \|D^3\bm{u}\|_{L^2(\Omega)} \leqslant C\alpha^{-2}(\|\psi\|_{\bm{H}^1(\Omega)} + \|\bm{u}\|_{L^2(\Omega)}),
     \end{split}
\end{equation}
 where $C$ is a constant depends only on $\Omega$.
\end{lemma}

\subsection{Main results}
Let $\bm{u}_0$ be a smooth, divergence-free vector field and satisfying non-penetration boundary conditions. Let $\bar{\bm{u}}$ be the corresponding smooth solution of the Euler equations \eqref{euler-1}--\eqref{euler-5}, of which the existence was proved by Kikuchi in \cite{kikuchi1983exterior}. We assume that the initial velocity $\bm{u}_0^{\nu,\alpha}$ of second grade fluid equations belongs to $\bm{H}^3 \cap V$. It was shown in our previous work \cite{xiaoguang2021secondgrade} that there exists a unique weak solution $\bm{u}^\alpha$ of \eqref{second-grade-origin-1}--\eqref{second-grade-origin-5}. We are now ready to state our main result.
\begin{theorem}\label{theorem-1} Let $T > 0$ be arbitrarily fixed. Assume that we are given a family of approximations $\{\bm{u}_0^{\nu,\alpha}\} \subset \bm{H}^3(\Omega) \cap V$ for $\bm{u}_0$ satisfying the  hypothesis \eqref{velocity-condition}, then we have
\begin{align}\label{thm-error-estimate}
\|\bm{u}^{\nu, \alpha} - \bar{\bm{u}}\|_{L^\infty([0, T]; L^2(\Omega))}  \leq K_T(\|\bm{u}_0^{\nu, \alpha} - \bm{u}_0\|_{L^2(\Omega)} + \alpha\|\nabla\bm{u}_0^{\alpha}\|_{L^2(\Omega)} + \alpha^\frac{1}{3} + \nu^{\frac{1}{2}}\alpha^{-\frac{2}{3}})
\end{align}
as $\alpha \rightarrow 0$. It follows that $\bm{u}^{\nu, \alpha}$ converges to $\bar{\bm{u}}$ in $C([0, T]; L^2(\Omega))$ under the constraint condition $\nu = o(\alpha^\frac{4}{3})$.
\end{theorem}

\begin{remark} The second author with his collaborators in \cite{lopes2015approximation} proved that, in bounded domain, the solutions to the second grade equations converges to the one of Euler equations in $C([0, T]; L^2(\Omega))$ space under the constraint condition that $\nu = O(\alpha^2)$. However, this constraint could be relaxed as $\nu = O(\alpha^\frac{4}{3})$, since we could obtain \eqref{thm-error-estimate} in like wise for bounded domain,  following the same energy method used in the next section.
\end{remark}

As we known, the second grade equations \eqref{second-grade-origin-1}--\eqref{second-grade-origin-5} becomes to be Euler-$\alpha$ equations if the viscosity $\nu = 0$. The above theorem thus immediately implies the following result.
\begin{corollary} Fix $T > 0$. Let $\bm{u}^{\alpha}$ be the solution to the Euler-$\alpha$ equations with initial data $\bm{u}^\alpha_0$ satisfying  hypothesis \eqref{velocity-condition}. Then $\bm{u}^\alpha$  converges to $\bar{\bm{u}}$   in the following sense:
\begin{align}
\|\bm{u}^{ \alpha} - \bar{\bm{u}}\|_{L^\infty([0, T]; L^2(\Omega))}  \leq K_T(\|\bm{u}_0^{\alpha} - \bm{u}_0\|_{L^2(\Omega)} + \alpha\|\nabla\bm{u}_0^{\alpha}\|_{L^2(\Omega)} + \alpha^\frac{1}{3} )
\end{align}
as $\alpha \rightarrow 0$.
\end{corollary}

\section{Proof of Theorem \ref{theorem-1}}
Before estimating the error term $\bm{u}^{\nu, \alpha} - \bar{\bm{u}}$(denoted as $\bm{w}^{\nu, \alpha}$), we need to solve the following issues. On one hand, observing that \eqref{second-grade-origin-1} has rather involved form of the nonlinearities and high order derivatives, thus a priori estimate of $\bm{u}^{\nu, \alpha}$ is indispensable. On the other hand, due to the mismatch between the non-slip boundary condition of second grade equations and the non-penetration boundary condition of Euler equations, we need to construct a boundary layer artificially.

We therefore divided this section into three subsections. In the first subsection, we will make use of vorticity-stream formula to obtain uniform estimates of $\bm{u}^{\nu, \alpha}$. In the second subsection, we will use stream function of $\bar{\bm{u}}$ to construct a boundary layer corrector $\bm{u}_b$. Moreover, some properties of $\bm{u}_b$ will be presented. In the last subsection, combing the uniform bound of $\bm{u}^{\nu, \alpha}$ and the boundary layer corrector,  we can estimate the error term $\bm{w}^{\nu, \alpha}$  by standard energy method.
\subsection{Estimate of $\bm{u}^{\nu, \alpha}$}

\begin{proposition}\label{proposition-estimate-u} For arbitrary fixed $T > 0$. Suppose the initial velocity $\bm{u}_0^{\nu,\alpha}$ satisfies  hypothesis \eqref{velocity-condition}, then the solution $\bm{u}^{\nu, \alpha}$ of the second grade equations \eqref{second-grade-origin-1}--\eqref{second-grade-origin-5} satisfies
\begin{align}\label{estimate-of-u-alpha}
\sup_{t\in[0, T]}\|D^k\bm{u}^{\nu, \alpha}\|_{L^2(\Omega)}^2 \leq K\alpha^{-k} \text{ for } k = 1, 2, 3.
\end{align}
\end{proposition}
\begin{proof}
We begin to prove \eqref{estimate-of-u-alpha} for the case that $k = 1, 2$. Indeed, multiply equation \eqref{second-grade-origin-1} for $\bm{u}^{\nu, \alpha}$ by $\bm{u}^{\nu, \alpha}$ and integrate over $[0, t) \times \Omega$, then we obtain the following estimates:
\begin{align}\label{low-uniform-bound-u}
\|\bm{u}^{\nu, \alpha}\|_{L^2(\Omega)}^2 + \alpha^2\|\nabla \bm{u}^{\nu, \alpha}\|_{L^2(\Omega)}^2 + \nu \int_0^t \|\nabla \bm{u}^{\nu, \alpha}\|_{L^2(\Omega)}^2 = \|\bm{u}_0^{\nu,\alpha}\|_{L^2(\Omega)}^2 + \alpha^2\|\nabla\bm{u}_0^{\nu,\alpha}\|_{L^2(\Omega)}^2.
\end{align}
Collecting with hypothesis \eqref{velocity-condition}, it follows that \eqref{estimate-of-u-alpha} holds for $k = 1, 2$.

It's somewhat more complex to prove  \eqref{estimate-of-u-alpha} for the case $k = 3$, for which we need to pick up some details about establishing well-posedness of second grade fluid equaitons in our previous work \cite{xiaoguang2021secondgrade}.

Let $\vartheta$ be a smooth function in $\mathbb{R}^2$ such that $\vartheta \equiv 1$ for $|x| < \frac{1}{2}$,
$\vartheta \equiv 0$ for $|x| > 1$ and $0 \leq \vartheta \leq 1$ in $\mathbb{R}^2$, set $\vartheta_n = \vartheta_n(x) = \vartheta(\frac{x}{n})$. Firstly, in Proposition 2.1 of \cite{xiaoguang2021secondgrade}, the authors showed there exists an unique solution $\bm{u}^n \in L^\infty([0, T]; \bm{H}^3(\Omega))$ to the following approximate equations.
\begin{numcases}{}
  \partial_t \bm{v}^n + \bm{u}^n \cdot {\nabla} \bm{v}^n + (\nabla \bm{u})^t \cdot \bm{v}^n + \nu \bm{v}^n - \frac{\nu}{\alpha^2} \vartheta_{n} \bm{u}^n + \nabla p^n = 0   &  $\text{in}  \ \Omega \times (0, T), $ \label{second-approximate-1}  \qquad \    \\
    \text{div} \ \bm{u}^n = 0 &  $\text{in} \ \Omega \times [0, T) $ \label{second-approximate-2},\\
    \bm{u}^n = 0  & $\text{in} \ \Gamma \times [0, T] \label{second-approximate-3}$,\\
    \bm{u}^n|_{t=0} = \bm{u}_0^n & $\text{in} \ \Omega $\label{second-approximate-4},\\
    \bm{u}^n (x, t) \to 0 & $\forall t \in [0, T), |x| \to \infty $ \label{second-approximate-5}.
\end{numcases}
where $\bm{v}^n = \bm{u}^n - \alpha^2 \Delta \bm{u}^n$ and $\{\bm{u}_0^n\}_{n\in\mathbb{N}}$ is a family of approximations for $\bm{u}_0^{\nu, \alpha}$ such that $\bm{u}_0^n$ converges to $\bm{u}_0^{\nu,\alpha}$ in $\bm{H}^3(\Omega)$.

Then, in the proof of Theorem 2.1 of \cite{xiaoguang2021secondgrade}, they showed that there exists an unique pair $(q^n, \psi^n) \in C([0, T]; L^2(\Omega)) \times C([0, T];\dot{\bm{H}}(\Omega))$ satisfying
\begin{numcases}{}
     \partial_t q^n + \bm{u}^n \cdot \nabla{q^n} + \frac{\nu}{\alpha^2} q^n -\frac{\nu}{\alpha^2} \nabla^\perp \cdot (\vartheta_{n} \bm{u}^n) = 0 & $ \text{in } \,\Omega \times [0, T]$ ,\label{approximate-1} \\
     q^n|_{t=0} = q_0^n & $ \text{in } \,  \Omega$, \label{approximate-2}\\
     \Delta_x {\psi^n}(x, t) = {q^n}(x, t) & $ \text{in } \,\Omega \times [0, T]$, \label{approximate-3} \\
     {\psi^n}(x, t) = 0 & $ \text{on } \, \Gamma \times [0, T]$, \label{approximate-4} \\
     {{\bm{u}}}^n(x, t) + \alpha^2 A{\bm{u}}^n(x, t) = \nabla^{\perp}{\psi^n}(x, t) &  $\text{in } \Omega \times [0, T]$, \label{approximate-5}  \\
     {\bm{u}}^n(x, t) = 0 & $\text{on } \Gamma \times [0, T]$ \label{approximate-6}.
\end{numcases}
where $q_0^n = \nabla^\perp \cdot (\bm{u}_0^n - \alpha^2\Delta \bm{u}_0^n)$ and $A$ is the Stokes operator. From the above equations $\eqref{approximate-1}$--\eqref{approximate-6}, they proved that the bound of $\bm{H}^3(\Omega)$-norm of $\bm{u}^n$ is independent of $n$. By Banach-Alaoglu theorem, they could find a subsequence of $\bm{u}^n$ converges weak-star in $L^\infty([0, T]; \bm{H}^3(\Omega))$ to some  $\bm{u}^{\nu, \alpha}$. Moreover, it was verified that $\bm{u}^n$ converges to $\bm{u}^{\nu, \alpha}$ strongly in $C([0, T]; \bm{H}^1(\Omega^{\prime}))$ for arbitrary compact subset $\Omega^\prime$ of $\Omega$. Therefore, $\bm{u}^{\nu, \alpha}$ was proved to be  the weak solution to the original second grade equations \eqref{second-grade-origin-1}--\eqref{second-grade-origin-5}.

In their proof of well-posedness of second grade fluid equations, the length scale parameter $\alpha$ and the viscosity $\nu$ was skipped in estimating the uniform bound of $\bm{u}^{\nu, \alpha}$. However in our discussion, $\alpha$ and $\nu$ both need to tend to zero, therefore it is necessary to expound the relationship between the uniform bound of $\bm{u}^{\nu, \alpha}$ and parameters $\alpha$ and $\nu$. 

Let us first consider the $\bm{H}^1(\Omega)$ norm of $\bm{u}^n$. We multiply \eqref{second-approximate-1} by $\bm{u}^n$ and integrate over $\Omega$, then obtain
\begin{align}
\nonumber\frac{1}{2}\frac{\dif} {\dif t}(\|\bm{u}^n\|_{L^2(\Omega)}^2 + \alpha^2 \|\nabla\bm{u}^n\|_{L^2(\Omega)}^2) + \nu\|\nabla\bm{u}^n\|_{L^2(\Omega)}^2 &= -\frac{\nu}{\alpha^2}(1-\|\vartheta_{n}\|_{L^\infty(\Omega)})\|\bm{u}^n\|_{L^2(\Omega)}^2  \leq 0.
\end{align}
the above inequality immediately yields
\begin{align}\label{u-n-1}
\|\bm{u}^n(t)\|_{L^2(\Omega)}^2 + \alpha^2 \|\nabla\bm{u}^n(t)\|_{L^2(\Omega)}^2 + \nu\int_0^t\|\nabla\bm{u}^n\|_{L^2(\Omega)}^2 \dif s &\leq \|\bm{u}_0^n\|_{L^2(\Omega)}^2 + \alpha^2 \|\nabla\bm{u}_0^n\|_{L^2(\Omega)}^2
\end{align}
for $t \in [0, T]$.

To obtain estimates of high order derivatives, we need to probe deep into equations \eqref{approximate-1}--\eqref{approximate-6}. Let us begin to consider the transport equations \eqref{approximate-1}--\eqref{approximate-2}. Multiplying \eqref{approximate-1} by $q^n$ and integrating over $\Omega \times [0, t]$, we find that
\begin{align}
\frac{1}{2}\| q^n\|_{L^2(\Omega)}^2 - \frac{1}{2}\|q^n_0\|_{L^2(\Omega)}^2 + \frac{\nu}{\alpha^2}\int_0^t \|q^n\|_{L^2(\Omega)}^2 \dif s - \frac{\nu}{\alpha^2}\int_0^t \int_\Omega \nabla\cdot(\vartheta_n\bm{u}^n) q^n \dif s = 0
\end{align}
Using Young's inequality, the above equation implies
\begin{align}\label{transport-estimate}
         \|q^n(t)\|_{L^2({\Omega})}^2 \leq \|q_0^n\|_{L^2({\Omega})}^2 + K\frac{\nu}{\alpha^2}\int_0^t \|\nabla\cdot(\vartheta_n\bm{u}^n)\|_{{L}^2(\Omega)}^2 \dif s
\end{align}
for $t \in [0, T]$.  Let $t \in [0, T]$ be fixed, we then concern the Poisson equations \eqref{approximate-3}--\eqref{approximate-4}. It follows from Lemma 2 in \cite{xiaoguang2021secondgrade} that
\begin{align}\label{poisson-estimate}
\|D^2\psi^n(t)\|_{L^2(\Omega)} \leq K(\|q^n(t)\|_{L^2(\Omega)} + \|\nabla \psi^n(t)\|_{L^2(\Omega)}).
\end{align}
We now consider equations \eqref{approximate-5}--\eqref{approximate-6}, which is the stationary Stokes equations.  Lemma $\ref{lemma-stokes}$ brings out that
\begin{align}\label{stokes-estimate-tmp}
 \|D^3\bm{u}^n(t)\|_{L^2(\Omega)} \leq K\alpha^{-2}(\|\nabla \psi^n(t)\|_{H^{1}(\Omega)} + \|\bm{u}^n(t)\|_{L^2(\Omega)}).
\end{align}
Collecting inequalities \eqref{poisson-estimate} and \eqref{stokes-estimate-tmp}, it follows
\begin{align}
\nonumber \ \|D^3\bm{u}^n(t)\|_{L^2(\Omega)} &\leq K\alpha^{-2}(\|q^n\|_{L^2(\Omega)} + \|\nabla\psi^n\|_{L^2(\Omega)} + \|\bm{u}^n\|_{L^2(\Omega)}).
\end{align}
Recalling equation \eqref{approximate-5}, we use Gagliardo-Nirenberg interpolation inequality and H\"older inequality to deduce that
\begin{align} \label{approximate-tmp-1}
  \|D^3\bm{u}^n(t)\|_{L^2(\Omega)} &\leq K\alpha^{-2}(\|q^n\|_{L^2(\Omega)} +  \|\bm{u}^n\|_{L^2(\Omega)}).
\end{align}
Collecting \eqref{u-n-1}, \eqref{transport-estimate} and \eqref{approximate-tmp-1} gives
\begin{align}\label{u-n-3-tmp}
\nonumber \|D^3\bm{u}^n(t)\|_{L^2(\Omega)} \leq K(&\|D^3\bm{u}_0^n\|_{L^2(\Omega)} + \alpha^{-2}\|\nabla\bm{u}_0^n\|_{L^2(\Omega)} + \alpha^{-3} \|\bm{u}_0^n\|_{L^2(\Omega)} \\
&+ \nu^\frac{1}{2}\alpha^{-1}n^{-1}\|\nabla\vartheta\|_{L^\infty(\mathbb{R}^2)}\left[\int_0^t \|\bm{u}^n\|_{L^2(\Omega)}^2\dif s\right]^\frac{1}{2} ).
\end{align}
where we have used the identity  $q_0^n = \nabla^\perp \cdot (\bm{u}_0^n - \alpha^2\Delta \bm{u}_0^n)$.
Let $n \rightarrow \infty$, we observe that $\bm{u}_0^n$ converges strongly to $\bm{u}_0^{\nu,\alpha}$ in $\bm{H}^3(\Omega)$ and $\bm{u}^n$ is uniformly bounded and converges weak-star to $\bm{u}^{\nu, \alpha}$ in $L^\infty([0, T]; \bm{H}^3(\Omega))$,  therefore \eqref{u-n-3-tmp} yields:
\begin{align}\label{uniform-bound-u}
\sup_{t\in[0, T]}\|D^3\bm{u}^{\nu, \alpha}(t)\|_{L^2(\Omega)} \leq K(\|D^3\bm{u}_0^{\nu,\alpha}\|_{L^2(\Omega)} + \alpha^{-2}\|\nabla \bm{u}_0^{\nu,\alpha}\|_{L^2(\Omega)}+ \alpha^{-3}\|\bm{u}_0^{\nu,\alpha}\|_{L^2(\Omega)}).
\end{align}
Combing with hypothesis \eqref{velocity-condition}, we find that \eqref{estimate-of-u-alpha} holds for $k = 3$. The case $k = 2$ immediately follows from Gagliardo-Nirenberg interpolation inequality.
\end{proof}

\subsection{Boundary layer corrector}
As in the statement of the theorem \ref{theorem-1}, $\bar{\bm{u}}$ denotes the smooth Euler equation in $\Omega$. From Lemma \ref{stream-function-lemma}, we knows that there exists a stream function $\bar{\psi}$ of $\bar{\bm{u}}$ such that $\bar{\bm{u}} = \nabla^\perp \bar{\psi}$ and $\bar{\psi}$ vanishes on $\Gamma$. Let $\eta: \mathbb{R}^+ \rightarrow [0, 1]$ be a smooth cut-off function such that
\begin{align}
\nonumber &\eta(x) = 1 \text{ for } x \in [0, 1],  \quad \eta(x) = 0 \text{ for } x \in [2, \infty).
\end{align}
Let $\delta > 0$ be a small number determined later, we define the boundary layer corrector $\bm{u}_b$ as
\begin{align}
\nonumber \bm{u}_b \equiv \bm{u}_b(x) := \nabla^\perp(\eta(\frac{\rho(x)}{\delta})\bar{\psi}),
\end{align}
where $\rho(x)$ is the distance from $x$ to the boundary $\Gamma$. For the corrector $\bm{u}_b$, we have the following result.
\begin{lemma}\label{corrector-lemma} Fix $T > 0$. There exists a constant $K$ such that for any $0 < \delta < 1$ and any $0 \leq t < T$ we have
\begin{itemize}
    \item[(i)] $\|\partial_t^l \bm{u}_b\|_{L^2(\Omega)} \leq K\delta^{\frac{1}{2}}$,
    \item[(ii)] $\|\partial_t^l \nabla \bm{u}_b\|_{L^2(\Omega)} \leq K\delta^{-\frac{1}{2}}$.
\end{itemize}
\end{lemma}
\begin{proof}
First we rewrite
\begin{align}\label{corrector-1}
\partial_t^l \bm{u}_b =  \delta^{-1}(\nabla^\perp\eta)(\frac{\rho(x)}{\delta})\partial_t^l\bar{\psi} + \eta(\frac{\rho(x)}{\delta})\partial_t^l\bar{\bm{u}}
\end{align}
The support of the first term of the right hand side of the above relation is contained in the annulus $\delta < |\rho(x)| < 2\delta$, whose Lebesgue measure is $O(\delta)$. Observing that $\bar{\psi}$ vanishes on $\Gamma$.  it follows from Poincare inequality that
\begin{align}\label{corrector-2}
\|(\nabla^\perp\eta)(\frac{\rho(x)}{\delta})\partial_t^l\bar{\psi}\|_{L^2(\Omega)} \leq K\delta^{\frac{3}{2}}
\end{align}
On the other hand, the  second term of the right hand side of \eqref{corrector-1} is supported in $\{x\in \Omega \ | \ \rho(x) < 2\delta \}$, noticing that $\bar{\bm{u}}$ is smooth, we find that
\begin{align}\label{corrector-3}
\| \eta(\frac{\rho(x)}{\delta})\partial_t^l\bar{\bm{u}}\|_{L^2(\Omega)} \leq K\delta^\frac{1}{2}
\end{align}
By substitution of \eqref{corrector-2} and \eqref{corrector-3} into \eqref{corrector-1}, we conclude that item (i) holds. Item (ii) can be proved similarily. Indeed, $\partial_t^l \partial_{x_i}\bm{u}_b$ can be expanded as
\begin{align}
\nonumber \partial_t^l \partial_{x_i} \bm{u}_b = &\delta^{-2}(\nabla^\perp_x\partial_{x_i}\eta)(\frac{\rho(x)}{\delta})\partial_t^l\bar{\psi} + \delta^{-1}(\nabla^\perp_x\eta)(\frac{\rho(x)}{\delta})\partial_t^l\partial_{x_i}\bar{\psi} \\
&+\delta^{-1} \partial_{x_i}\eta(\frac{\rho(x)}{\delta})\bar{\bm{u}} + \eta(\frac{\rho(x)}{\delta})\partial_t^l\partial_{x_i}\bar{\bm{u}}
\end{align}
We will verify each term on the right side of the above expansion. Observing that $\eta$ is supported in $B(0, 2)$ and $\partial_t^l \bar{\psi}$ vanishes on $\Gamma$, it follows that the first term can be estimated by Poincare inequality:
\begin{align}\label{corrector-tmp-1}
\|\delta^{-2}(\nabla^\perp_x\partial_{x_i}\eta)(\frac{\rho(x)}{\delta})\partial_t^l\bar{\psi}\|_{L^2(\Omega)} \leq \delta^{-2}\|D^2\eta\|_{L^\infty(\mathbb{R}^2)} \|\partial_t^l\bar{\psi}\|_{L^2(\Omega_{\delta})} \leq K\delta^{-\frac{1}{2}},
\end{align}
where $\Omega_{\delta} = \{x \in \Omega \ | \ \rho(x) \leq 2\delta\}$. Notice that $\bar{\psi}$ is smooth and $\eta$ is a cut function, it follows that the second term  can be estimated as
\begin{align}\label{corrector-tmp-2}
 \|\delta^{-1}(\nabla^\perp_x\eta)(\frac{\rho(x)}{\delta})\partial_t^l\partial_{x_i}\bar{\psi}\|_{L^2(\Omega)} \leq \delta^{-1}\|\nabla\eta\|_{L^\infty(\mathbb{R}^2)}\|\partial_t^l\partial_{x_i}\bar{\psi}\|_{L^2(\Omega_\delta)} &\leq K\delta^{-\frac{1}{2}}.
\end{align}
The third term can be handled similarily. Indeed, since $\bar{\bm{u}}$ is smooth, it follows
\begin{align}
\|\delta^{-1} \partial_{x_i}\eta(\frac{\rho(x)}{\delta})\bar{\bm{u}}\|_{L^2(\Omega)} \leq \delta^{-1}\|\nabla\eta\|_{L^\infty(\mathbb{R}^2)}\|\partial_t^l\bar{\bm{u}}\|_{L^2(\Omega_\delta)} &\leq K\delta^{-\frac{1}{2}}.
\end{align}
Using again the property that $\eta$ is supported in $B(0, 2)$, the last term can be estimated as
\begin{align}\label{corrector-tmp-3}
\|\eta(\frac{\rho(x)}{\delta})\partial_t^l\partial_{x_i}\bar{\bm{u}}\|_{L^2(\Omega)} \leq K\delta^{-\frac{1}{2}}.
\end{align}
Collecting estimates \eqref{corrector-tmp-1}--\eqref{corrector-tmp-3}, we conclude that item (ii) holds. The proof is therefore completed.
\end{proof}

\subsection{Energy estimates of $\bm{w}^{\nu, \alpha}$}
Recall that $\bm{w}^{\nu, \alpha}$ is defined as $\bm{u}^{\nu, \alpha} - \bar{\bm{u}}$, therefore $\bm{w}^{\nu, \alpha}$ is divergence free. We now perform an energy estimate, subtracting \eqref{euler-1} from \eqref{second-grade-origin-1}, multiplying $\bm{w}^{\nu, \alpha}$ and integrating over $[0, t] \times \Omega$ for $t \in [0, T]$, we obtain
\begin{align}\label{energe-estimate-identity}
\nonumber \frac{1}{2}\| \bm{w}^{\nu, \alpha}\|_{L^2(\Omega)}^2 &- \frac{1}{2}\|\bm{w}^{\nu, \alpha}_0\|_{L^2(\Omega)}^2 =\nu\int_0^t \int_\Omega \Delta \bm{u}^{\nu, \alpha} \cdot \bm{w}^{\nu, \alpha} \dif x \dif s\\
\nonumber& \ -\int_0^t \int_\Omega  \left[(\bm{w}^{\nu, \alpha} \cdot \nabla)\bar{\bm{u}}\right] \cdot \bm{w}^{\nu, \alpha} \dif x \dif s + \alpha^2\int_0^t\int_\Omega \partial_t \Delta \bm{u}^{\nu, \alpha} \cdot \bm{w}^{\nu, \alpha} \dif x \dif s\\
\nonumber &\ + \alpha^2 \int_0^t\int_\Omega \left[(\bm{u}^{\nu, \alpha}\cdot \nabla)\Delta \bm{u}^{\nu, \alpha} + (\nabla \bm{u}^{\nu, \alpha})^t \cdot \Delta \bm{u}^{\nu, \alpha}\right]\cdot \bm{w}^{\nu, \alpha}  \dif x \dif s\\
 &:= I_1 + I_2 + I_3 + I_4,
\end{align}
where $\bm{w}_0^{\nu, \alpha} = \bm{u}_0^{\nu, \alpha} - \bm{u}_0$. We will examine each one of the four terms on the right hand-side of above identify. We look at first term, which can be divided to 2 parts:
\begin{align}
\nonumber I_1 &= \nu \int_0^t\int_\Omega \Delta \bm{u}^{\nu, \alpha} \cdot \bm{w}^{\nu, \alpha} \dif x \dif s\\
\nonumber &= \nu \int_0^t\int_\Omega \Delta \bm{u}^{\nu, \alpha} \cdot \bm{u}^{\nu, \alpha} - \nu\int_0^t\int_\Omega \bm{u}^{\nu, \alpha} \cdot \bar{\bm{u}} \dif x \dif s\\
\nonumber &:= I_{11} + I_{12}.
\end{align}
The first part $I_{11}$ is a good part and is estimated as
\begin{align}
\nonumber I_{11} = -\nu \int_0^t \|\nabla\bm{u}^{\nu, \alpha}\|_{L^2(\Omega)}^2.
\end{align}
We then consider the second part $I_{12}$. By integration by parts, $I_{12}$ can be rewritten as
\begin{align}
\nonumber I_{12} &= -\nu \int_0^t \int_\Omega \Delta \bm{u}^{\nu, \alpha} \cdot (\bar{\bm{u}} - \bm{u}_b) \dif x \dif s - \nu \int_0^t \int_\Omega \Delta \bm{u}^{\nu, \alpha} \cdot \bm{u}_b \dif x \dif s\\
\nonumber &= \nu \int_0^t \int_\Omega \nabla \bm{u}^{\nu, \alpha} \cdot \nabla (\bar{\bm{u}} - \bm{u}_b) \dif x \dif s - \nu \int_0^t \int_\Omega \Delta \bm{u}^{\nu, \alpha} \cdot \bm{u}_b \dif s \dif s,
\end{align}
using  Cauchy-Schwarz inequality and Young's inequality, it follows from the properties of the boundary corrector $\bm{u}_b$ in Lemma \ref{corrector-lemma} that
\begin{align}
\nonumber I_{12} &\leq \nu \int_0^t \|\nabla \bm{u}^{\nu, \alpha}\|_{L^2(\Omega)} \|\nabla(\bar{\bm{u}} - \bm{u}_b)\|_{L^2(\Omega)} \dif x \dif s + \nu \int_0^t \|\Delta \bm{u}^{\nu, \alpha}\|_{L^2(\Omega)} \|\bm{u}_b\|_{L^2(\Omega)} \dif s\\
\nonumber &\leq \frac{\nu}{2}\int_0^t \|\nabla\bm{u}^{\nu, \alpha}\|_{L^2(\Omega)}^2 \dif s + K_T\nu\delta^{\frac{1}{2}}\sup_{t\in[0, T]}\|\Delta\bm{u}^{\nu, \alpha}\|_{L^2(\Omega)} + K\nu \delta^{-1}.
\end{align}
Collecting the estimates of $I_{11}$ and $I_{12}$, we conclude that
\begin{align}
 I_1 \leq  -\frac{\nu}{2} \int_0^t \|\nabla\bm{u}^{\nu, \alpha}\|_{L^2(\Omega)}^2 + K_T\nu\delta^{\frac{1}{2}}\sup_{t\in[0, T]} \|\Delta\bm{u}^{\nu, \alpha}\|_{L^2(\Omega)} + K\nu \delta^{-1}.
\end{align}
Next, we concern the second term $I_2$. By Cauchy-Schwarz inequality, it follows that
\begin{align}
\nonumber I_2 &= -\int_0^t \int_\Omega  \left[(\bm{w}^{\nu, \alpha} \cdot \nabla)\bar{\bm{u}}\right] \cdot \bm{w}^{\nu, \alpha} \dif x \dif s \\
\nonumber &\leq \int_0^t \|\nabla \bar{\bm{u}}\|_{L^\infty(\Omega)} \|\bm{w}^{\nu, \alpha}\|_{L^2(\Omega)}^2 \dif s\\
\nonumber &\leq K\int_0^t \|\bm{w}^{\nu, \alpha}\|_{L^2(\Omega)}^2.
\end{align}
We now begin to estimate $I_3$, which can be divided into two parts
\begin{align}
\nonumber I_3 &= \alpha^2\int_0^t\int_\Omega \partial_t \Delta \bm{u}^{\nu, \alpha} \cdot \bm{w}^{\nu, \alpha} \dif x \dif s\\
\nonumber &= \alpha^2 \int_0^t\int_\Omega \partial_t \Delta \bm{u}^{\nu, \alpha} \bm{u}^{\nu, \alpha} \dif x \dif s - \alpha^2\int_0^t \int_\Omega \partial_t \Delta \bm{u}^{\nu, \alpha} \bar{\bm{u}} \dif x \dif s\\
\nonumber &:=I_{31} + I_{32}.
\end{align}
The first part $I_{31}$ is a good one, and it's easy to see that
\begin{align}
\nonumber I_{31} &= \alpha^2 \int_0^t\int_\Omega \partial_t \Delta \bm{u}^{\nu, \alpha} \bm{u}^{\nu, \alpha} \dif x \dif s\\
\nonumber&=-\frac{\alpha^2}{2}(\|\nabla\bm{u}^{\nu, \alpha})\|_{L^2(\Omega)}^2 - \|\nabla\bm{u}^{\alpha}_0\|_{L^2(\Omega)}^2).
\end{align}
The second part $I_{32}$ is a bad term. Through integration by parts, we obtain
\begin{align}
\nonumber I_{32} &=  -\alpha^2\int_0^t \int_\Omega \partial_t \Delta \bm{u}^{\nu, \alpha} \cdot \bar{\bm{u}} \dif x \dif s\\
\nonumber&= -\alpha^2 \int_0^t\int_\Omega \partial_t \Delta\bm{u}^{\nu, \alpha} \cdot (\bar{\bm{u}} - \bm{u}_b) \dif x \dif s -\alpha^2 \int_0^t \int_\Omega \partial_t \Delta\bm{u}^{\nu, \alpha} \cdot \bm{u}_b \dif x \dif s\\
\nonumber&=\alpha^2\int_\Omega \nabla \bm{u}^{\nu, \alpha} \cdot \left[\nabla (\bar{\bm{u}}-\bm{u}_b)\right] \dif x - \alpha^2 \int_\Omega \nabla \bm{u}^{\alpha}_0 \cdot \left[\nabla (\bar{\bm{u}}_0-\bm{u}_b|_{t=0})\right] \dif x \\
\nonumber&- \alpha^2 \int_0^t \int_\Omega \nabla \bm{u}^{\nu, \alpha} \cdot \left[\partial_t\nabla (\bar{\bm{u}}-\bm{u}_b)\right] \dif x \dif s + \alpha^2 \int_0^t\int_\Omega \Delta \bm{u}^{\nu, \alpha} \cdot \partial_t \bm{u}_b \dif x \dif s\\
\nonumber&- \alpha^2 \int_\Omega \Delta \bm{u}^{\nu, \alpha} \cdot \bm{u}_b \dif x + \alpha^2 \int_\Omega \Delta \bm{u}_0^{\alpha} \cdot \bm{u}_b|_{t=0} \dif x,
\end{align}
In view of the properties of the corrector $\bm{u}_b$ in Lemma \ref{corrector-lemma}, it follows from Cauchy-Schwarz inequality and Young's inequality that
\begin{align}
\nonumber I_{32} &\leq K\alpha^2\delta^{-\frac{1}{2}}(\int_0^t \|\nabla \bm{u}^{\nu, \alpha}(s)\|_{L^2(\Omega)} \dif s + \|\nabla \bm{u}^{\nu, \alpha}\|_{L^2(\Omega)} + \|\nabla \bm{u}^{\alpha}_0\|_{L^2(\Omega)})\\
\nonumber&+ K\alpha^2\delta^{\frac{1}{2}}( \int_0^t \|\Delta \bm{u}^{\nu, \alpha}(s)\|_{L^2(\Omega)} \dif s + \|\Delta \bm{u}^{\nu, \alpha}\|_{L^2(\Omega)} + \|\Delta \bm{u}_0^{\nu,\alpha}\|_{L^2(\Omega)})\\
\nonumber&\leq \frac{\alpha^2}{4} \|\nabla \bm{u}^{\nu, \alpha}\|_{L^2(\Omega)}^2 + K\alpha^2\int_0^t \|\nabla \bm{u}^{\nu, \alpha}(s)\|_{L^2(\Omega)}^2 \dif s +  K\alpha^2 \|\nabla \bm{u}^{\alpha}_0\|_{L^2(\Omega)}^2\\
\nonumber&+ K_T\alpha^2\delta^{\frac{1}{2}}\sup_{t\in[0, T]} \|\Delta \bm{u}^{\nu, \alpha}\|_{L^2(\Omega)}+ K_T\alpha^2\delta^{-1}.
\end{align}
Combining the bounds for $I_{31}$ and $I_{32}$ gives that
\begin{align}
\nonumber I_3 &\leq -\frac{\alpha^2}{4} \|\nabla \bm{u}^{\nu, \alpha}\|_{L^2(\Omega)}^2 + K\alpha^2\int_0^t \|\nabla \bm{u}^{\nu, \alpha}(s)\|_{L^2(\Omega)}^2 \dif s +  K\alpha^2 \|\nabla \bm{u}^{\alpha}_0\|_{L^2(\Omega)}^2\\
\nonumber&+ K_T\alpha^2\delta^{\frac{1}{2}}\sup_{t\in[0, T]} \|\Delta \bm{u}^{\nu, \alpha}\|_{L^2(\Omega)}+ K_T\alpha^2\delta^{-1}.
\end{align}
We are left to estimate $I_4$. By integration by parts, $I_4$ is divided into two parts
\begin{align}
\nonumber I_4 &= -\alpha^2 \int_0^t\int_\Omega \left[(\bm{u}^{\nu, \alpha}\cdot \nabla)\Delta \bm{u}^{\nu, \alpha} + (\nabla \bm{u}^{\nu, \alpha})^t \cdot \Delta \bm{u}^{\nu, \alpha}\right]\cdot W  \dif x \dif s\\
\nonumber&= -\alpha^2 \int_0^t\int_\Omega \left[(\bm{u}^{\nu, \alpha}\cdot \nabla)\Delta \bm{u}^{\nu, \alpha} + (\nabla \bm{u}^{\nu, \alpha})^t \cdot \Delta \bm{u}^{\nu, \alpha}\right]\cdot \bar{\bm{u}}  \dif x \dif s\\
\nonumber&=-\alpha^2 \int_0^t\int_\Omega \left[(\bm{u}^{\nu, \alpha}\cdot \nabla)\Delta \bm{u}^{\nu, \alpha}\right] \cdot \bar{\bm{u}} \dif x \dif s  - \alpha^2 \int_0^t\int_\Omega \left[(\nabla \bm{u}^{\nu, \alpha})^t \cdot \Delta \bm{u}^{\nu, \alpha}\right]\cdot \bar{\bm{u}}  \dif x \dif s\\
\nonumber&:= I_{41} + I_{42},
\end{align}
using integration by parts again, the first part $I_{41}$ can be rewritten as
\begin{align}
\nonumber I_{41} &= -\alpha^2 \int_0^t\int_\Omega \left[(\bm{u}^{\nu, \alpha}\cdot \nabla)\Delta \bm{u}^{\nu, \alpha}\right] \cdot \bar{\bm{u}} \dif x \dif s\\
\nonumber &= \alpha^2 \int_0^t\int_\Omega \left[(\bm{u}^{\nu, \alpha}\cdot \nabla)\bar{\bm{u}}\right] \cdot \Delta \bm{u}^{\nu, \alpha} \dif x \dif s\\
\nonumber &= -\alpha^2 \int_0^t\int_\Omega \partial_i \bm{u}^{\nu, \alpha}_j \partial_j\bar{\bm{u}}_k \partial_i \bm{u}^{\nu, \alpha}_k \dif x \dif s -\alpha^2 \int_0^t\int_\Omega \bm{u}^{\nu, \alpha}_j \partial_i \partial_j\bar{\bm{u}}_k \partial_i \bm{u}^{\nu, \alpha}_k \dif x \dif s .
\end{align}
Recalling the estimates of $\bm{u}^{\nu, \alpha}$ in Proposition \ref{proposition-estimate-u} and the properties of $\bm{u}_b$ established in Lemma \ref{corrector-lemma}, we deduce that
\begin{align}
\nonumber I_{41} &\leq \alpha^2\int_0^t \|\nabla\bm{u}^{\nu, \alpha}\|_{L^2}(\Omega)^2 \|\bar{\bm{u}}\|_{L^\infty(\Omega)} \dif s + \alpha^2\int_0^t \|\nabla \bm{u}^{\nu, \alpha}\|_{L^2(\Omega)} \|\bm{u}^{\nu, \alpha}\|_{L^2(\Omega)} \|D^2\bar{\bm{u}}\|_{L^\infty(\Omega)}\\
\nonumber &\leq K\alpha^2 \int_0^t \|\nabla\bm{u}^{\nu, \alpha}\|_{L^2(\Omega)}^2 \dif s + K\alpha^2.
\end{align}
Next, we begin to estimate $I_{42}$. Through integration by parts, we arrive at
\begin{align}
\nonumber I_{42} &= - \alpha^2 \int_0^t\int_\Omega \left[(\nabla \bm{u}^{\nu, \alpha})^t \cdot \Delta \bm{u}^{\nu, \alpha}\right]\cdot \bar{\bm{u}}  \dif x \dif s\\
\nonumber &= \alpha^2 \int_0^t\int_\Omega \left[(\bar{\bm{u}} \cdot \nabla) \Delta \bm{u}^{\nu, \alpha}\right]\cdot \bm{u}^{\nu, \alpha} \dif x \dif s\\
\nonumber &=-\alpha^2\int_0^t \int_\Omega  \partial_i \bar{\bm{u}}_j \partial_i \partial_j \bm{u}^{\nu, \alpha}_k \bm{u}^{\nu, \alpha}_k - \alpha^2 \int_0^t \int_\Omega \bar{\bm{u}}_j \partial_i \partial_j \bm{u}^{\nu, \alpha}_k \partial_i \bm{u}^{\nu, \alpha}_k\\
\nonumber &= \alpha^2 \int_0^t \int_\Omega \partial_i \bar{\bm{u}}_j \partial_i\bm{u}^{\nu, \alpha} \partial_j \bm{u}^{\nu, \alpha}.
\end{align}
Then, using the properties of the corrector $\bm{u}_b$, we deduce that
\begin{align}
\nonumber I_{42} \leq K\alpha^2 \int_0^t \|\nabla\bm{u}^{\nu, \alpha}\|_{L^2(\Omega)}^2 \dif s.
\end{align}
Combing the estimates of $I_{41}$ and $I_{42}$ yields
\begin{align}
\nonumber I_4 \leq K\alpha^2 \int_0^t \|\nabla\bm{u}^{\nu, \alpha}\|_{L^2(\Omega)}^2 \dif s + K\alpha^2.
\end{align}
Now, we substitute the estimates of $I_i$, $i=1, ..., 4$ into the identity $\eqref{energe-estimate-identity}$, it follows that
\begin{align}\label{energy-estimate-inequality}
\nonumber\frac{1}{2}\| \bm{w}^{\nu, \alpha}\|_{L^2(\Omega)}^2 &+ \frac{\alpha^2}{4}(\|\nabla\bm{u}^{\nu, \alpha})\|_{L^2(\Omega)}^2 + \frac{\nu}{2}\int_0^t \|\nabla\bm{u}^{\nu, \alpha}\|_{L^2(\Omega)}^2 \dif s \leq  K\int_0^t \|\bm{w}^{\nu, \alpha}\|_{L^2(\Omega)}^2 \\
 & + K\alpha^2 \int_0^t \|\nabla\bm{u}^{\nu, \alpha}\|_{L^2(\Omega)}^2 \dif s  + \frac{1}{2}\|\bm{w}^{\nu, \alpha}_0\|_{L^2(\Omega)}^2 + K\alpha^2\|\nabla\bm{u}_0^{\alpha}\|_{L^2(\Omega)}^2 +g(\nu, \alpha, \delta)
\end{align}
where
\begin{align}
\nonumber g(\nu, \alpha, \delta) = K_T(\nu + \alpha^2)\delta^{\frac{1}{2}}\sup_{t\in[0, T]} \|\Delta \bm{u}^{\nu, \alpha}\|_{L^2(\Omega)} + K_T(\nu+\alpha^2)\delta^{-1} + K\alpha^2.
\end{align}
We apply the bound of $\bm{u}^{\nu, \alpha}$ in Proposition \ref{proposition-estimate-u} to deduce that
\begin{align}
\nonumber g(\nu, \alpha, \delta) \leq K_T(\nu + \alpha^2)(\alpha^{-2} \delta^{\frac{1}{2}} + \delta^{-1}) + K\alpha^2
\end{align}
To minimize $g$, we choose $\delta = \alpha^\frac{4}{3}$. Then the above inequality implies
\begin{align}\label{g-estimate}
 g(\nu, \alpha, \delta) \leq K_T(\alpha^\frac{2}{3} +\nu\alpha^{-\frac{4}{3}}).
\end{align}
Collecting \eqref{energy-estimate-inequality} and \eqref{g-estimate}, it follows from Gronwall inequality that
\begin{align}\label{energy-estimate}
\sup_{t\in[0, T]}\left[\|\bm{w}^{\nu, \alpha}\|_{L^2(\Omega)}^2 + \frac{\alpha^2}{2} \|\nabla\bm{u}^{\nu, \alpha}\|_{L^2(\Omega)}^2\right] \leq K_T(\|\bm{w}_0^{\nu, \alpha}\|_{L^2(\Omega)}^2 + \alpha^2\|\nabla\bm{u}_0^{\alpha}\|_{L^2(\Omega)}^2 + \alpha^\frac{2}{3} + \nu\alpha^{-\frac{4}{3}})
\end{align}
Recalling the hypothesis (H), we conclude that
\begin{align}
\sup_{t\in[0, T]}\left[\|\bm{w}^{\nu, \alpha}\|_{L^2(\Omega)}^2 + \frac{\alpha^2}{2} \|\nabla\bm{u}^{\nu, \alpha}\|_{L^2(\Omega)}^2\right] \rightarrow 0 \text{ as } \alpha \rightarrow 0
\end{align}
provided that $\nu = o(\alpha^\frac{4}{3})$. This completes the proof.

\section{Examination of hypothesis \eqref{velocity-condition}}
In this section, we will examine the hypothesis \eqref{velocity-condition} for which we can prove Theorem \ref{theorem-1}. For bounded domain, the authors \cite{lopes2015convergence} constructed a suitable family of approximations for $\bm{u}_0$ using the spectral theory of Stokes operator. In this article, however, we will make use of the stream function of velocity field. We assume the initial velocity of Euler fluid is smooth, divergence free and satisfies non-penetration boundary condition. With the help of Lemma \ref{stream-function-lemma}, we have the following proposition.
\begin{proposition} \label{approximate-initial-velocity-lemma} Suppose that $\bm{u}_0 \in \bm{H}^3(\Omega) \cap L^2_\sigma$, then there exists a family $\{\bm{u}_0^{\nu,\alpha}\} \subset \bm{H}^3(\Omega) \cap V$ satisfying the hypothesis \eqref{velocity-condition}.
\end{proposition}
\begin{proof}
Let $\psi$ be the stream function of $\bm{u}_0$ such that $\psi \equiv 0$ on $\Gamma$, the existence of which is verified in Lemma \ref{stream-function-lemma}. Let $\varphi \in C^\infty(\mathbb{R})$ be a smooth function such that $\varphi(x) \equiv 1$ for $|x| > 2$, $\varphi(x) \equiv 0$ for $|x| < 1$ and $0 \leq \varphi(x) \leq 1$ in $\mathbb{R}$. We define $\bm{u}^\alpha_0$ as
\begin{align}
\bm{u}^\alpha_0(x) := \nabla^\perp\left(\varphi\left(\frac{\rho(x)}{\alpha}\right) \psi(x)\right) \quad \forall x \in \Omega,
\end{align}
where $\rho(x)$ is the distance from $x$ to boundary $\Gamma$. It's easy to see that $\bm{u}_0^{\nu,\alpha} \in \bm{H}^3(\Omega) \cap V$. To prove item (i) of the hypothesis \eqref{velocity-condition}, we set $\Pi_\alpha := \{ x\in \Omega \big| \rho(x) < 2\alpha\}$, it follows:
\begin{align}\label{diff-equation-1}
\nonumber \|\bm{u}_0^{\nu,\alpha} - \bm{u}_0\|_{L^2(\Omega)} &= \|\nabla^\perp\left(\varphi\left(\frac{\rho(\cdot)}{\alpha}\right)\right) \psi+ \left(\varphi\left(\frac{\rho(\cdot)}{\alpha}\right) -1\right) \bm{u}_0\|_{L^2(\Omega)}\\
 &\leq K\alpha^{-1}\|\nabla\varphi\|_{L^\infty(\mathbb{R}^2)}\|\psi\|_{L^2(\Pi_\alpha)} + \|\bm{u}_0\|_{L^2(\Pi_\alpha)}
\end{align}
We then consider each term on the right side of the above inequality. Observing that $\psi \equiv 0$ on $\Gamma$ and $\varphi$ is a smooth cut function, we have
\begin{align}
\nonumber\|\nabla \varphi\|_{L^\infty(\mathbb{R}^2)}\|\psi\|_{L^2(\Pi_\alpha)} \leq K\alpha^{\frac{3}{2}}.
\end{align}
On the other hand, since $\bm{u}_0 \in \bm{H}^3(\Omega) \subset L^\infty(\Omega)$, we deduce that
\begin{align}
 \nonumber \|\bm{u}_0\|_{L^2(\Pi_\alpha)} \leq K\alpha^{\frac{1}{2}}.
\end{align}
Substituting the above inequalities into \eqref{diff-equation-1}, we arrive at
\begin{align}\label{u_0-approximate-0-estimate}
\|\bm{u}_0^{\nu,\alpha} - \bm{u}_0\|_{L^2(\Omega)} \leq K\alpha^{\frac{1}{2}}
\end{align}
therefore item (i) is verified.

We shall now prove item (ii) of hypothesis \eqref{velocity-condition}. Indeed, for $ 1 \leq k \leq 3$,  we find that
\begin{align} \label{diff-equation-2}
 \nonumber \|D^k(\bm{u}_0^{\nu,\alpha})\|_{L^2(\Omega)} &\leq \|D^{k+1}\left(\varphi\left(\frac{\rho(\cdot)}{\alpha}\right)\right) \psi(\cdot)\|_{L^2(\Pi_\alpha)} + \|\varphi D^{k}\bm{u}_0\|_{L^2(\Omega)} \\
 & + \sum_{1 \leq |\beta|\leq k, |\gamma|\leq k}C_{\beta, \gamma}\|D^\beta\left(\varphi\left(\frac{\rho(\cdot)}{\alpha}\right)\right) D^{\gamma}\bm{u}_0\|_{L^2(\Pi_\alpha)}.
\end{align}
We will examine each term on the right side of the above inequality. Notice again that $\psi \equiv 0$ on $\Gamma$ and $\varphi$ is smooth in $\mathbb{R}^2$, the first term is estimated as
\begin{align}
\nonumber \|D^{k+1}\left(\varphi\left(\frac{\rho(\cdot)}{\alpha}\right)\right) \psi(\cdot)\|_{L^2(\Pi_\alpha)} &\leq K\alpha^{-k-1}\|\nabla^{k+1}\varphi\|_{L^\infty(\mathbb{R}^2)}\|\psi\|_{L^2(\Pi_\alpha)}\\
\nonumber &\leq K\alpha^{-k+\frac{1}{2}}.
\end{align}
Since $\bm{u}_0 \in \bm{H}^3(\Omega)$, it is easy to see that the second term obeys
\begin{align}
\nonumber \|\varphi D^{k}\bm{u}_0\|_{L^2(\Omega)} \leq \|\bm{u}_0\|_{\bm{H}^3(\Omega)} \leq K.
\end{align}
Recalling that $K^\beta\varphi$ is compactly supported in $B(0, 2)$, it follows that the third term can be estimated as
\begin{align}
\nonumber\sum_{1 \leq |\beta|\leq k, |\gamma|\leq k}C_{\beta, \gamma}\|D^\beta\left(\varphi\left(\frac{\rho(\cdot)}{\alpha}\right)\right) D^{\gamma}\bm{u}_0\|_{L^2(\Pi_\alpha)} \leq K\alpha^{-k + \frac{1}{2}}.
\end{align}
Substituting the above three inequalties into \eqref{diff-equation-2}, we conclude that
\begin{align}\label{u_0-approximate-k-estimate}
 \|D^k\bm{u}_0^{\nu,\alpha}\|_{L^2(\Omega)} \leq K\alpha^{-k+ \frac{1}{2}},
\end{align}
therefore (ii) is verified. The proof of Lemma \ref{approximate-initial-velocity-lemma} is completed.
\end{proof}

Notice that, comparing with hypothesis (H), the approximate family $\bm{u}_0^{\nu, \alpha}$ for $\bm{u}_0$ constructed in Proposition \ref{approximate-initial-velocity-lemma} tells more information. In fact, collecting \eqref{energy-estimate}, \eqref{u_0-approximate-0-estimate} and \eqref{u_0-approximate-k-estimate}, we obtain the following result immediately.
\begin{corollary} Let $T > 0$ be fixed. Let $\bar{\bm{u}}$ be the solution to Euler equations \eqref{euler-1}--\eqref{euler-5}, and $\bm{u}^{\nu, \alpha}$ be the solution of second grade fluid equations \eqref{second-grade-origin-1}--\eqref{second-grade-origin-5} with the initial velocity $\bm{u}_0^{\nu, \alpha}$ constructed in Proposition \ref{approximate-initial-velocity-lemma}. Then $\bm{u}^{\nu, \alpha}$ converges to $\bar{\bm{u}}$ as follows
\begin{align}
\sup_{t \in [0,T]} \left[\|\bm{w}^{\nu, \alpha}\|_{L^2(\Omega)} + \alpha \|\nabla\bm{u}^{\nu, \alpha}\|_{L^2(\Omega)}\right] \leq K_{T}(\alpha^{\frac{1}{3}} + \nu\alpha^{-\frac{4}{3}}).
\end{align}
In particular case $\nu = 0$, we have the Euler-$\alpha$ equations converges to Euler equation at rate $\alpha^\frac{1}{3}$.
\end{corollary}

\acknowledgements{\rm The work of Aibin Zang was supported in part  by the National Natural Science Foundation of China (Grant no. 11771382, 12061080).}

\bibliographystyle{siam}
\bibliography{vanishing_limit}

\end{document}